\documentclass[12pt, reqno]{amsart}
\usepackage{amsfonts,latexsym,amsthm,amssymb,amsmath,amscd,euscript,bm}
\usepackage[margin = 2cm]{geometry}
\usepackage{color}
\usepackage{hyperref}
\usepackage{url}
\usepackage{breakurl}
\usepackage{float}
\newcommand{\bburl}[1]{\textcolor{blue}{\url{#1}}}

\newcommand{\dfn}[1]{\textcolor{blue}{\textit{#1}}}

\usepackage{thmtools}
\usepackage{thm-restate}

\newtheorem{thm}{Theorem}[section]
\newtheorem{lem}[thm]{Lemma}

\theoremstyle{definition}
\newtheorem{defi}[thm]{Definition}

\theoremstyle{remark}

\newcommand\ben{\begin{enumerate}}
\newcommand\een{\end{enumerate}}

\newcommand{\pre}{\mathrm{pre}}
\newcommand{\nc}{\newcommand}

\nc{\N}{\mathbb N}
\nc{\DD}{\mathbb D}
\nc{\TT}{\mathbb T}
\nc{\EE}{\mathbb E}

\nc{\cT}{\mathcal T}
\nc{\cP}{\mathcal P}
\nc{\cM}{\mathcal M}
\nc{\cC}{\mathcal C}
\nc{\cB}{\mathcal B}
\nc{\cG}{\mathcal G}
\nc{\cD}{\mathcal D}
\nc{\cA}{\mathcal A}
\nc{\cS}{\mathcal S}
\nc{\cF}{\mathcal F}
\nc{\cL}{\mathcal L}
\nc{\cR}{\mathcal R}
\nc{\be}{\mathbf{e}}
\nc{\bv}{\mathbf{v}}
\nc{\ba}{\mathbf{a}}
\nc{\vc}{\mathbf{c}}

\numberwithin{equation}{section}
\numberwithin{part}{section}

\title
{
    A note on multiset reconstruction from sum and pairwise products
}

\author[Li]{Stella Jiahui Li}
\email{\textcolor{blue}{\href{mailto:stellali@college.harvard.edu}{stellali@college.harvard.edu}}}
\address{Department of Mathematics, Harvard University}

\begin{document}

\begin{abstract}
    Ballantine, Beck, and Merca defined a map $\pre_2$, which sends an integer partition $\lambda = (\lambda_1, \dots, \lambda_{\ell})$ to the set $\{\lambda_i\lambda_j : 1 \leq i < j \leq \ell\}$
    consisting of the pairwise products of parts of $\lambda$. The same three authors and Sagan conjectured that for each $n$, the map $\pre_2$ is injective on the set of integer partitions of $n$. In this note, we prove their conjecture.
\end{abstract}

\maketitle

\section{Background}

An \dfn{integer partition} $\lambda = (\lambda_1, \dots, \lambda_\ell)$ of a positive integer $n$ is a weakly decreasing sequence of positive integers  whose sum is $n$. We refer to the $\lambda_i$'s as the \dfn{parts} of a partition, $|\lambda| = n$ as the \dfn{size}, and $\ell(\lambda) = \ell$ as the \dfn{length}. 

The $j$th \dfn{elementary symmetric polynomial} $e_j$ is defined by 
    $$e_j(x_1, \dots, x_n) = \begin{cases}
        \sum_{1 \leq i_1 < i_2 < \cdots < i_j \leq n} x_{i_1} \cdots x_{i_j} &\text{if $j \leq n$}; \\
        0 &\text{if }j > n.
    \end{cases}$$

We are interested in the partition whose parts are the summands of $e_j(\lambda_1, \dots, \lambda_{\ell})$. 

\begin{defi}[\cite{BBM24}]
    Given a partition $\lambda = (\lambda_1, \dots, \lambda_{\ell})$, we define $\pre_k(\lambda)$ to be the partition whose multiset of parts is $\{\lambda_{i_1}\cdots \lambda_{i_k} \ : 1 \leq i_1 < i_2 < \cdots < i_k \leq \ell \}$. If $\ell(\lambda)  < j$, then $\pre_k(\lambda)$ is the empty partition $\emptyset$. We call $\pre_k(\lambda)$ an \dfn{elementary symmetric partition}.
\end{defi}

For example, consider the partition $\lambda = (4, 2, 1, 1)$ and $k = 2$. We can evaluate $e_2(4,2,1,1)$ as $$e_2(4,2,1,1) = 4\cdot 2 + 4 \cdot 1+ 4\cdot 1 + 2 \cdot 1 + 2 \cdot 1+1 \cdot 1,$$ and by looking at the summands, we get that $\pre_2(\lambda) = (8,4,4,2,2,1)$.

The main conjecture from \cite{BBM24} and \cite{BBMS24+} is that the map $\pre_2$ is injective on the set of partitions of $n$ for each $n$. Our main result is a proof of this conjecture. 

\begin{thm}\label{thm: main theorem} If $\lambda$ and $\mu$ are partitions of a positive integer $n$, then $\pre_2(\lambda) = \pre_2(\mu)$ if and only if $\lambda = \mu$. 
\end{thm}

In fact, our proof establishes a stronger statement where we replace partitions of $n$ by finite nonincreasing sequences of positive real numbers summing to $n$.

\section{Proof}

We start with an auxiliary lemma. 

\begin{lem}\label{lem: largest term}
    Let $\lambda$ and $\mu$ be integer partitions of a positive integer $n$. If $\lambda_1 = \mu_1$ and $\pre_2(\lambda) = \pre_2(\mu)$, then $\lambda = \mu$.
\end{lem}
\begin{proof}
    Observe that $\ell(\lambda) = \ell(\mu)$ because $\ell(\pre_2(\lambda)) = \binom{\ell(\lambda)}{2} = \binom{\ell(\mu)}{2} = \ell(\pre_2(\mu))$. We want to show $\lambda_i = \mu_i$ for all $1 \leq i \leq \ell(\lambda) = \ell(\mu)$, and we proceed by strong induction on $i$. For the base case, we know $\lambda_1 = \mu_1$ by assumption.

    For the inductive step, observe that 
    \begin{align*}
        \lambda_1 \lambda_k &= \max\left(\pre_2(\lambda) \setminus \{\!\!\{\lambda_i\lambda_j: 1 \leq i < j \leq k-1\}\!\!\}\right) \\
        &= \max\left(\pre_2(\mu) \setminus \{\!\!\{\mu_i\mu_j: 1 \leq i < j \leq k-1\}\!\!\}\right) = \mu_1\mu_k,
    \end{align*}
    where $\{\!\!\{\}\!\!\}$ denotes multiset. Since $\lambda_1 = \mu_1$, we conclude that $\lambda_k = \mu_k$.  
\end{proof}

With Lemma \ref{lem: largest term} established, we have the necessary tool to prove our main result.

\begin{proof}[Proof of Theorem \ref{thm: main theorem}]
    It suffices to show the identity $\sum_{1\leq i \leq \ell} \lambda_i^{2^m} = \sum_{1\leq i \leq \ell} \mu_i^{2^m}$ for all positive integers $m$ because then
    $$\lambda_1 = \lim_{m \to \infty} \left(\sum_{1\leq i \leq \ell} \lambda_i^{2^m}\right)^{1/2^m} = \lim_{m \to \infty} \left(\sum_{1\leq i \leq \ell} \mu_i^{2^m}\right)^{1/2^m} = \mu_1,$$ 
    and we can apply Lemma \ref{lem: largest term}.
    
    We prove the identity by strong induction on $m$. For the base case, $\sum_{i} \lambda_i = \sum_{i}\mu_i = n$ by assumption. For the inductive step, observe that
    \begin{align*}
        \sum_{1\leq i \leq \ell(\lambda)} \lambda_i^{2^m} &= \left(\sum_{1\leq i \leq \ell(\lambda)} \lambda_i^{2^{m-1}} \right)^2 - 2\sum_{\lambda_i\lambda_j \in \pre_2(\lambda)}(\lambda_i\lambda_j)^{2^{m-1}} \\
        &= \left(\sum_{1\leq i \leq \ell(\mu)} \mu_i^{2^{m-1}} \right)^2 - 2\sum_{\mu_i\mu_j \in \pre_2(\mu)}(\mu_i\mu_j)^{2^{m-1}} = \sum_{1\leq i \leq \ell(\mu)} \mu_i^{2^m},
    \end{align*}
    as desired.
\end{proof}

We remark that a classical result of Selfridge and Strauss \cite{SS58} states that the set $X = \{x_1, \dots, x_{\ell}\} \subset \mathbb{C}$ is uniquely determined by the set of pairwise sums of elements of $X$ if and only if $\ell$ is not a power of $2$. An immediate consequence of this theorem is that $\pre_2$ is injective on partitions with lengths that are not powers of $2$, regardless of the sizes of the partitions. The strength of our result is that we can also handle partitions with lengths that are powers of $2$ if we fix their sizes. 

\section*{Acknowledgments}

This research was conducted at the University of Minnesota Duluth REU with support from Jane Street Capital, NSF Grant $2409861$, and donations from Ray Sidney and Eric Wepsic. I would like to thank Mitchell Lee, Noah Kravitz, and Katherine Tung for their invaluable guidance throughout the research process and helpful suggestions during the editing process. I am very grateful to Joe Gallian and Colin Defant for organizing the Duluth REU and for the invitation to participate.

\bibliography{biblio}

\end{document}